\documentclass[12pt,twoside]{amsart}
\usepackage{a4wide, amsmath,amsbsy,amsfonts,amssymb,mathabx,
stmaryrd,amsthm,mathrsfs,graphicx, amscd, tikz-cd, cancel}
\usepackage[normalem]{ulem}
\usepackage{soul}
\usetikzlibrary{matrix,arrows,decorations.pathmorphing}
\usepackage[active]{srcltx}
\usepackage[
	hypertexnames=false,
	hyperindex,
	pagebackref,  
	pdftex,
	breaklinks=true,
	bookmarks=false,
	colorlinks,
	linkcolor=blue,
	citecolor=red,
	urlcolor=red,
]{hyperref}

\usepackage[all]{xy}
\usepackage{subfig}
\usepackage{tikz}
\usetikzlibrary{shapes,arrows,shadows}
\usetikzlibrary{decorations.markings}
\usepackage{color}
\usepackage[normalem]{ulem}
\usepackage{hyperref}
\usepackage{wrapfig}
\usetikzlibrary{arrows}
\usepackage{pinlabel}
\long\def\symbolfootnote[#1]#2{\begingroup%
\def\thefootnote{\fnsymbol{footnote}}\footnote[#1]{#2}\endgroup}

\newtheorem{innercustomthm}{Theorem}
\newenvironment{customtheorem}[1]
  {\renewcommand\theinnercustomthm{#1}\innercustomthm}
  {\endinnercustomthm}

\newtheorem{theorem}{Theorem}[section]

\newtheorem{lemma}[theorem]{Lemma}

\theoremstyle{definition}

\newtheorem{remark}[theorem]{Remark}

\newtheorem*{namedtheorem}{\theoremname}
\newcommand{\theoremname}{testing}

\DeclareMathSymbol{\Alpha}{\mathalpha}{operators}{"41}
\DeclareMathSymbol{\Beta}{\mathalpha}{operators}{"42}
\DeclareMathSymbol{\Epsilon}{\mathalpha}{operators}{"45}
\DeclareMathSymbol{\Zeta}{\mathalpha}{operators}{"5A}
\DeclareMathSymbol{\Eta}{\mathalpha}{operators}{"48}
\DeclareMathSymbol{\Iota}{\mathalpha}{operators}{"49}
\DeclareMathSymbol{\Kappa}{\mathalpha}{operators}{"4B}
\DeclareMathSymbol{\Mu}{\mathalpha}{operators}{"4D}
\DeclareMathSymbol{\Nu}{\mathalpha}{operators}{"4E}
\DeclareMathSymbol{\Omicron}{\mathalpha}{operators}{"4F}
\DeclareMathSymbol{\Rho}{\mathalpha}{operators}{"50}
\DeclareMathSymbol{\Tau}{\mathalpha}{operators}{"54}
\DeclareMathSymbol{\Chi}{\mathalpha}{operators}{"58}
\DeclareMathSymbol{\omicron}{\mathord}{letters}{"6F}

\newcommand{\N}{\mathbb{N}}
\newcommand{\Z}{\mathbb{Z}}

\def\Aut{\operatorname{Aut}}
\def\Sp{\operatorname{Sp}}

\def\Out{\operatorname{Out}}

\def\PG{\mathrm{P}\Gamma}

\def\cI{{\mathcal I}}

\def\cE{{\mathcal E}}

\def\cK{{\mathcal K}}

\def\td{\tilde}

\def\g{{\gamma}}
\def\G{{\Gamma}}
\def\wG{\widetilde{\Gamma}}

\def\ssm{\smallsetminus}
\def\ol{\overline}

\def\da{\downarrow}
\def\sr{\stackrel}

\def\Ra{\Rightarrow}

\def\co{\colon\thinspace}

\begin{document}
\title[A generating set for the Johnson kernel]{A generating set for the Johnson kernel}

\author{Marco Boggi}
\address{UFF - Instituto de Matem\'atica e Estat\'{\i}stica -
Niter\'oi - RJ 24210-200; Brazil}
\email{marco.boggi@gmail.com}

\begin{abstract}For a connected orientable hyperbolic surface $S$ without boundary and of finite topological type, the Johnson kernel $\cK(S)$ 
is the subgroup of the mapping class group of $S$ generated by Dehn twists about separating simple closed curves on $S$. 
We prove that $\cK(S)$ is generated by the Dehn twists about separating simple closed curves on $S$ bounding either: 
a closed subsurface of genus $1$ or $2$; a closed subsurface of genus $1$ minus one point; a closed disc minus two points.
\vskip 0.2cm
\noindent AMS Mathematics Subject Classification: 57K20; 57K31; 57M07.
\end{abstract}

\maketitle

\section{Introduction}
Let $S=S_{g,n}$ be a closed orientable differentiable surface of genus $g(S)=g$ from which $n(S)=n$ points have been removed. 
We assume that the Euler characteristic $\chi(S)=2-2g-n$ of $S$ is negative. We let $S_g:=S_{g,0}$.
Let then $\G(S)$ and $\PG(S)$ be respectively the mapping class group and the pure mapping class group of the surface $S$.
We will sometimes use the notation $\G_{g,[n]}:=\G(S_{g,n})$ and $\G_{g,n}:=\PG(S_{g,n})$.
The \emph{Johnson kernel} (or \emph{Johnson subgroup}) $\cK(S)$ is the subgroup of $\G(S)$ (and $\PG(S)$) generated by
Dehn twists about separating simple closed curves on $S$. We also let $\cK_{g,n}:=\cK(S_{g,n})$.

Classically (cf.\ \cite{J1}), for a closed surface $S$ or a $1$-pointed surface $(S,P)$ of genus $\geq 2$, the Johnson kernel 
is defined to be the kernel (whence the name) of the natural representation $\G(S)\to\Out(\pi_1(S,P)/\pi_1(S,P)^{[3]})$
(resp.\ $\G(S,P)\to\Aut(\pi_1(S,P)/\pi_1(S,P)^{[3]})$), where, for a group $G$, we denote by $G^{[k]}$ the $k$-th term 
of the lower central series of $G$, i.e.\ $G^{[1]}:=G$ and $G^{[k+1]}:=[G^{[k]},G]$, for all $k\in\N^+$. In \cite[Theorems~5 and 6]{J3}, Johnson
then showed that these kernels coincide with $\cK(S)$ and $\cK(S\ssm P)$, respectively. 

More generally (cf.\ \cite[Theorem~3.13]{mod}), 
$\cK(S)$ is the kernel of the natural representation $\G(S)\to\Out(\pi_1(S,P)/W^3\pi_1(S,P))$, where $W^3\pi_1(S,P)$ is the third term of the
\emph{weight filtration} on the fundamental group of the (punctured) surface $S$ which is defined as follows (cf.\ \cite[Section~2]{mod}).
Let $N$ be the kernel of the natural epimorphism $\pi_1(S,P)\to\pi_1(\ol{S},P)$, where $\ol{S}$ is the closed surface obtained from $S$
filling in the punctures. Then, we define $W^1\pi_1(S,P):=\pi_1(S,P)$; $W^2\pi_1(S,P):=N\cdot[\pi_1(S,P),\pi_1(S,P)]$ and:
\[W^{k+1}\pi_1(S,P):=[W^k\pi_1(S,P),\pi_1(S,P)]\cdot[W^{k-1}\pi_1(S,P), N].\]

The \emph{Torelli group} $\cI(S)$ is the kernel of the natural representation $\G(S)\to\Sp(H_1(\ol{S};\Z))$, where
$\ol{S}$ is the closed surface obtained from $S$ filling in all the punctures. There is an inclusion $\cK(S)\subseteq\cI(S)$.
A deep result by Johnson (cf.\ \cite{J4}) then states that, for $n(S)\leq 1$, the commutator subgroup $[\cI(S),\cI(S)]$ is contained
in $\cK(S)$ as a finite index subgroup and that the (abelian) quotients $\cI(S)/\cK(S)$ and $\cK(S)/[\cI(S),\cI(S)]$ are, respectively, 
a finitely generated free $\Z$-module and a finite $\Z/2$-module.
This description also holds for the case $n(S)> 1$ (cf.\ \cite[Theorem~3.18 and Theorem~3.21]{mod}).

After being an open question for a long time, it is now settled that the Johnson kernel is finitely generated for $n(S)\leq 1$ and $g(S)\geq 4$
(cf.\ \cite[Theorem~1.1]{EH}, \cite[Theorem~A]{CEP}). In \cite[Theorem~1.3]{EF}, a finite (albeit complicated) generating set consisting 
of Dehn twists is given. In this paper, we address a related but different question, namely what topological types of Dehn twists are 
sufficient and necessary in a generating set for the Johnson kernel. 

This problem is motivated by \cite[Theorem~2]{J}, which states that, 
for $g(S)\geq 3$ and $n(S)\leq 1$, the Torelli group $\cI(S)$ is normally generated by a genus $1$ bounding pair map 
(cf.\ \cite[Section~1]{J} for the definition of such mapping class) and by a similar result for the Johnson kernel (cf.\ \cite[Theorem~1]{J}), 
which states that, for $g(S)\geq 2$ and $n(S)\leq 1$, the Johnson kernel $\cK(S)$ is generated by the Dehn twists about separating 
simple closed curves on $S$ bounding either a closed subsurface of genus $1$ or $2$. Here, we prove a more general version of Johnson's result:

\begin{customtheorem}{A}The Johnson kernel $\cK(S)$ is generated by the Dehn twists about separating simple closed curves on $S$
bounding either: a closed subsurface of genus $1$ or $2$; a closed subsurface of genus $1$ minus one point; a closed disc minus two points.
\end{customtheorem}

\begin{remark}The proof of Theorem~A, which we will give in the following section, does not make use of \cite[Theorem~1]{J}
and provides, in particular, a new proof of the closed surface case of this theorem.
\end{remark}

For  $g(S)\geq 2$ and $n(S)\leq 1$, we then show that the set of generators for $\cK(S)$ given by Johnson (cf.\  \cite[Theorem~1]{J}) is optimal 
in the following sense. Given a group $G$ and a normal subgroup $N$ of $G$, we say that a subset $X$ of $G$ is a \emph{normal generating set} 
for $N$ if the set $X$ is closed under conjugation by elements of $G$ and generates the subgroup $N$. We have:

\begin{customtheorem}{B}For $g(S)\geq 2$ and $n(S)\leq 1$, the normal generating set for the Johnson kernel $\cK(S)$ given by the 
Dehn twists about separating simple closed curves on $S$ bounding either a closed subsurface of genus $1$ or $2$ is minimal.
\end{customtheorem}
\bigskip

\noindent
{\bf Acknowledgements.} This manuscript was in part inspired by the conversations I had with Louis Funar during my visit at Institut Fourier, 
Universit\'e Grenoble Alpes in 2025. I thank him for the hospitality, the financial support and the conversations.

\section{Generating the Johnson kernel}
\subsection{A preliminary Lemma}
Let $\Pi_{g,n}:=\pi_1(S_{g,n},P)$ for some choice of base point $P\in S$.
We say that an element of $\Pi_{g,n}$ is \emph{simple} if it contains a simple closed curve.
For two (homotopy classes of) simple closed curves $\alpha$ and $\beta$ on $S$, let us denote by $i(\alpha,\beta)$
their geometric intersection number. 
This is well defined for two elements of the group $\Pi_{g,n}$,
once we fix (as we have done above) an identification of $\Pi_{g,n}$ with the fundamental group of $S$.
For the proof of Theorem~A, we will need the closed surface case of the following lemma:

\begin{lemma}\label{abelianquot}For $g(S)\geq 1$, let us define the following normal subgroup of $\Pi_{g,n}$:
\[N:=\langle\left.[\alpha,\beta]\right|\;\alpha,\beta\mbox{ simple elements of } \Pi_{g,n}\mbox{ such that }i(\alpha,\beta)=1\rangle.\]
The quotient $\Pi_{g,n}/N$ is then an abelian group.
\end{lemma}

\begin{proof}For simplicity, let us first assume that $n(S)\leq 1$. We will then explain how to modify the argument for the general case.
Since $g(S)\geq 1$ and $n(S)\leq 1$, there is a standard set of generators $\alpha_1,\beta_1,\dots,\alpha_g,\beta_g$ 
for the fundamental group $\Pi_{g,n}$ such that $i(\alpha_i,\beta_j)=\delta_{ij}$ and $i(\alpha_i,\alpha_j)=i(\beta_i,\beta_j)=0$, 
for all $1\leq i,j\leq g$.  

Let us denote by $\bar\alpha_i,\bar\beta_j$ the images of these elements in the quotient $\Pi_{g,n}/N$, for $1\leq i,j\leq g$.
By hypothesis, we already know that $\bar\alpha_i$ and $\bar\beta_i$ commute for all $i=1,\dots,g$. The conclusion follows if we show
that $\bar\alpha_i$ also commutes with $\bar\alpha_j$ and $\bar\beta_j$, for $i\neq j$.

It is easy to check that either $\alpha_i\cdot\alpha_j$ or $\alpha_i\cdot\alpha_j^{-1}$ (resp.\ $\alpha_i\cdot\beta_j$ or $\alpha_i\cdot\beta_j^{-1}$)
is a simple element. Let us then suppose that $\alpha_i\cdot\alpha_j$ (resp.\ $\alpha_i\cdot\beta_j$) has this property.
Since $i(\beta_j,\alpha_i\cdot\alpha_j)=1$ (resp.\ $i(\alpha_j,\alpha_i\cdot\beta_j)=1$, by hypothesis, we have that 
$\bar\beta_j$ commutes with $\bar\alpha_i\cdot\bar\alpha_j$ (resp.\ $\bar\alpha_j$ commutes with $\bar\alpha_i\cdot\bar\beta_j$).
It then follows that $\bar\beta_j$ commutes with $\bar\alpha_i$ (resp.\ $\bar\alpha_j$ commutes with $\bar\alpha_i$), as we had to prove.

For $n:=n(S)\geq 2$, in order to generate the group $\Pi_{g,n}$ we need to add to the above generating set simple elements 
$\g_1,\dots,\g_{n-1}$ such that $i(\g_i,\g_j)=0$, for all $1\leq i,j\leq n-1$, $i(\g_i,\beta_j)=0$, for all $i=1,\dots,n-1$ and 
$j=1,\dots,g$, $i(\g_i,\alpha_j)=0$, for all $i=1,\dots,n-1$ and $j=1,\dots,g-1$, but $i(\g_i,\alpha_g)=1$, for all $i=1,\dots,n-1$. 
Then, arguing essentially as above, we conclude that  the image of this generating set in the quotient $\Pi_{g,n}/N$ again
consists of commuting elements.
\end{proof}

\subsection{Proof of Theorem~A}
Theorem~A trivially holds for $g(S)=0$, $n(S)=3$, for $g(S)=1$, $n(S)\leq 2$ and for $g(S)\leq 5$, $n(S)=0$. 
For the proof of the general case, we then proceed by induction on $g(S)+n(S)$, where the base of the induction is provided 
by the cases $g(S)+n(S)\leq 3$. Let us denote by $(A_{g,n})$ the following statement:

\begin{enumerate}
\item[$(A_{g,n})\co$] Theorem~A holds for $S=S_{g,n}$.
\end{enumerate}

The proof of Theorem~A then reduces to the proof of the following two lemmas:

\begin{lemma}\label{(g,1)g+1}For $g\geq 5$, we have that $(A_{g,1})\Ra(A_{g+1,0})$.
\end{lemma}

\begin{proof}As it is customary, for $n(S)\leq 1$, we say that a separating simple closed curve $\g$ on $S$ and the 
associated Dehn twist $\tau_\g$ have genus $h$ if, for $n(S)=1$, the subsurface of $S$ without puncture or, for $n(S)=0$, the subsurface 
of smaller genus bounded by $\g$ has genus $h$. Note that in this terminology a separating simple closed curve on $S_{g,1}$ bounding
a $1$-punctured, genus $1$ closed subsurface of $S$ has genus $g-1$.

For every separating simple closed curve $\delta$ on $S_{g+1}$ of genus $>2$, we have 
to prove that the Dehn twist $\tau_\delta$ is a product of Dehn twists of genus $1$ and $2$. Let then $\g$ be a separating simple closed
curve on $S_{g+1}$ of genus $1$ disjoint from $\delta$. 

Let us denote by $S'$ and $S''$ the connected components of $S_{g+1}\ssm\g$.
The stabilizer $\G(S_{g+1})_\g$ of the isotopy class of $\g$ in the mapping class group $\G(S_{g+1})$ has the property that  
$\tau_\delta\in\G(S_{g+1})_\g$ and fits in the short exact sequence:
\[1\to\tau_\g^\Z\to\G(S_{g+1})_\g\to\G(S')\times\G(S'')\to 1.\]
Let us suppose that $\delta$ is contained in $S'$. We then have that $g(S')=g$ and $n(S')=1$ 
and, by hypothesis, the image of $\tau_\delta$ in $\G(S')$ is a product of Dehn twists of genus $1$, $2$ and $g-1$.
Let us observe now that a simple closed curve on $S'$ which bounds a $1$-punctured genus $1$ closed
subsurface of $S'$ identifies in $S_{g+1}$ with a simple closed curve bounding a genus $2$ subsurface of $S_{g+1}$.
Therefore, $\tau_\delta$ is a product of some power of $\tau_\g$ and a product of Dehn twists of genus $1$ and $2$.
\end{proof}

\begin{lemma}\label{(g,n)(g,n+1)}For $2g-2+n>0$, we have that $(A_{g,n})\Ra(A_{g,n+1})$.
\end{lemma}

\begin{proof}Let $\cE_{g,n}$ be the normal subgroup of $\G_{g,n}$ (and of the Johnson kernel $\cK_{g,n}$) 
generated by the Dehn twists about simple closed curves on $S$ which bound either a genus $1$ subsurface with at most one puncture, 
an unpunctured genus $2$ subsurface, or a $2$-punctured disc.

Let $\wG_{g,n}:=\G_{g,n}/\cK_{g,n}$ and $\wG_{g,n}':=\G_{g,n}/\cE_{g,n}$. Since $\cE_{g,n}\subseteq\cK_{g,n}$,
there is a natural epimorphism $\wG_{g,n}'\to\wG_{g,n}$.
The statement of Lemma~\ref{(g,n)(g,n+1)} can then be reformulated as the statement that, if the epimorphism $\wG_{g,n}'\to\wG_{g,n}$
is an isomorphism, the same is true for the epimorphism $\wG_{g,n+1}'\to\wG_{g,n+1}$.

Let us label by $P_1,\dots,P_{n+1}$ the punctures on $S_{g,n+1}$,
let $q\co\G_{g,n+1}\to\G_{g,n}$ be the epimorphism induced filling in the last puncture with a point which we also label by $P_{n+1}$ and
let $\Pi_{g,n}:=\pi_1(S_{g,n},P_{n+1})$. There is then a short exact sequence (the Birman exact sequence):
\begin{equation}\label{Birman}
1\to\Pi_{g,n}\sr{p}{\to}\G_{g,n+1}\sr{q}{\to}\G_{g,n}\to 1.
\end{equation}

By \cite[Theorem~3.12]{mod}, for $2g-2+n>0$, the Birman exact sequence~\eqref{Birman} induces on the quotients $\wG_{g,n}$ 
a short exact sequence:
\[1\to H_1(S_g;\Z)\sr{\tilde{p}}{\to}\wG_{g,n+1}\sr{\tilde{q}}{\to}\wG_{g,n}\to 1.\]
The Birman exact sequence~\eqref{Birman} also implies that, for $2g-2+n>0$, the quotients $\wG_{g,n}'$ fit in the exact sequence:
\[\Pi_{g,n}\sr{\bar p}{\to}\wG_{g,n+1}'\sr{\tilde{q}'}{\to}\wG_{g,n}'\to 1.\]

The push map $p\co\Pi_{g,n}\to\G_{g,n+1}$ maps a simple element of $\Pi_{g,n}$ which contains a simple closed curve 
bounding the puncture $P_i$ to a Dehn twist about a simple closed curve which bounds a disc containing only the punctures
labeled by $P_i$ and $P_{n+1}$, for $i=1,\dots,n$. By the definition of the subgroup $\cE_{g,n+1}$, the induced homomorphism 
$\bar p\co\Pi_{g,n}\to\wG_{g,n+1}'$ then factors through a homomorphism $\bar p'\co\Pi_{g}\to\wG_{g,n+1}'$, 
where we let $\Pi_g:=\pi_1(S_g,P_{n+1})$.

Let $\alpha,\beta\in\Pi_g$ be simple elements such that $i(\alpha,\beta)=1$. The commutator $[\alpha,\beta]$
contains a $P_{n+1}$-pointed oriented simple closed curve $\g$ on $S_g$ which bounds a genus $1$ subsurface of $S_g$ and
there is a $P_{n+1}$-pointed oriented simple closed curve $\td\g$ on $S_{g,n}$ which bounds an unpunctured genus $1$ 
subsurface of $S_{g,n}$ and whose image in $S_g$ is contained in the equivalence class $[\alpha,\beta]\in\Pi_g$.

The push map $p\co\Pi_{g,n}\to\G_{g,n+1}$ maps the equivalence class $[\td\g]$ of $\td\g$ in $\Pi_{g,n}$ to a product 
$\tau_{\td\g^+}\tau_{\td\g^-}^{-1}$, where one of the two simple closed curves $\td\g^+$ and $\td\g^-$ bounds an unpunctured genus $1$ 
subsurface of $S_{g,n+1}$ and the other a genus $1$ subsurface of $S_{g,n+1}$ with just one puncture labeled by $P_{n+1}$. 
From the definition of the subgroup $\cE_{g,n+1}$, it follows that $p([\td\g])\in\cE_{g,n+1}$, so that the image $\bar p'([\alpha,\beta])$ 
is trivial in the quotient $\wG_{g,n+1}'$.

By Lemma~\ref{abelianquot}, we then have that the homomorphism $\bar p'\co\Pi_{g}\to\wG_{g,n+1}'$ factors through
a homomorphism $\td{p}'\co H_1(S_g;\Z)\to\wG_{g,n+1}'$. In conclusion, the Birman exact sequence~\eqref{Birman} 
induces a commutative diagram with exact rows:
\begin{equation}\label{diagfund}\begin{array}{ccccccc}
&H_1(S_g;\Z)&\sr{\tilde{p}'}{\to}&\wG_{g,n+1}'&\sr{\tilde{q}'}{\to}&\wG_{g,n}'&\to 1\,\\
&\da&&\da&&\da&\\
1\to&H_1(S_g;\Z)&\sr{\tilde{p}}{\to}&\wG_{g,n+1}&\sr{\tilde{q}}{\to}&\wG_{g,n}&\to 1,
\end{array}\end{equation}
where the vertical maps are all surjective and, by hypothesis, the right hand map is an isomorphism. 
Since an epimorphism of a finitely generated free $\Z$-module to itself is an isomorphism, the left hand vertical map is an isomorphism as well. 
From the five lemma, applied to the commutative diagram~\eqref{diagfund}, it then follows that the middle map in the diagram is an isomorphism.
\end{proof}

\section{Minimality}
In this section, we will prove Theorem~B. For $g(S)=2$, $n(S)\leq 1$ and $g(S)=3$, $n(S)=0$, the statement of the theorem is trivial,
since all separating Dehn twists have genus $1$.
We will then consider first the case $g(S)\geq 4$, $n(S)\leq 1$ of the theorem and then the remaining case $g(S)=3$, $n(S)=1$.

\subsection{Proof of Theorem~B for $g(S)\geq 4$}The natural homomorphism $\cK_{g,1}\to\cK_g$ maps genus $1$
and genus $2$ Dehn twists to Dehn twists of the same genus for all  $g\geq 4$ (this is not true for $g=3$).
Hence, in this case, it is enough to prove the case $n=0$ of the theorem. 

Let us recall that in \cite[Section~5]{M1}, for $g\geq 2$, Morita defined a 
$\G_{g}$-equivariant homomorphism $d_0\co\cK_{g}\to\Z$, where the mapping class group acts by conjugation on the domain 
and trivially on the codomain, which maps a Dehn twist of genus $h$ to $h(g-h)\in\Z$ and so maps Dehn twists of genus $1$ to
$g-1$ and Dehn twists of genus $2$ to $2(g-2)$. 

In order to show that Dehn twists of genus $1$ cannot be removed from the given normal generating set, it is 
then enough to consider the reduction of the homomorphism $d_0$ modulo $g-2$ which sends Dehn twists of genus $1$ to
the element $\bar 1=\ol{g-1}\in\Z/(g-2)$ and Dehn twists of genus $2$ to the element $\bar 0=\ol{2(g-2)}\in\Z/(g-2)$. 

In order to show that Dehn twists of genus $2$ cannot be removed from the given normal generating set, we instead consider 
the reduction of the homomorphism $d_0$ modulo $g-1$ which sends Dehn twists of genus $1$ to the element $\bar 0\in\Z/(g-1)$ and 
Dehn twists of genus $2$ to the element $\bar 2=\ol{2(g-2)}\in\Z/(g-1)$ which is not zero for $g\geq 4$. 

\subsection{Proof of Theorem~B for $g(S)=3$, $n(S)=1$}Let $\G_g^1$, for $g\geq 2$, be the mapping class group of a closed surface
$S_g$ relative to a closed disc $D$ on $S_g$. There is a natural epimorphism $\G_g^1\to\G_{g,1}$ whose kernel is generated by the Dehn 
twist about the boundary of $D$ (which is a nontrivial element of $\G_g^1$). Let then $\cK_g^1$ be the subgroup of $\G_g^1$ generated 
by separating Dehn twists. In \cite[Section~5]{M1}, for $g\geq 2$, Morita also defines two $\G_{g}^1$-equivariant homomorphisms 
$d\co\cK_g^1\to\Z$ and $d'\co\cK_g^1\to\Z$, where the mapping class group acts by conjugation on the domain and trivially on 
the codomain, which map a Dehn twist of genus $h$, respectively, to $4h(h-1)$ and $h(2h+1)$.

For $g(S)=3$, the reduction of the homomorphism $d$ modulo $24$ then induces the $\G_{3,1}$-equivariant homomorphism
$\bar d\co\cK_{3,1}\to\Z/24$ which maps Dehn twists of genus $1$ to $\bar 0$ and Dehn twists of genus $2$ to $\bar 8$.
This shows that the Dehn twists of genus $2$ cannot be removed from the given normal generating set.

Let us then assume by contradiction that $\cK_{3,1}$ is generated by Dehn twists of genus $2$. This implies that $\cK_{3}^1$ is 
generated by Dehn twists of genus $2$ and $3$.
Let us  consider the direct sum homomorphism $ d\oplus d'\co\cK_3^1\to\Z\oplus\Z$ which maps Dehn twists of genus $1$
to $(0,3)$, Dehn twists of genus $2$ to $(8,10)$ and Dehn twists of genus $3$ to $(24,21)$.
A Dehn twist of genus $1$ can be expressed as a product of Dehn twists of genus $2$ and $3$ only if there exist 
$x,y\in\Z$ such that $(0,3)=x(8,10)+y(24,21)$. However, it is easy to check that the given system of linear equations has 
no integral solutions. This implies that $\cK_{3,1}$ is not generated by Dehn twists of genus $2$ and proves the theorem also in this case.


\begin{thebibliography}{99}

\bibitem{mod}M.~Boggi. \textsl{Fundamental groups of moduli stacks of stable curves of compact type.}
Geom. Topol. {\bf 13} (2009), 247--276. 


\bibitem{CEP}T. Church, M. Ershov, A. Putman. \textsl{On finite generation of the Johnson filtrations.} 
J. Eur. Math. Soc. {\bf 24}, no. 8 (2022), 2875--2914.

\bibitem{EF}M. Ershov, M. Franz. \textsl{Effective finite generation for $[\mathrm{IA}_n,\mathrm{IA}_n]$ and the Johnson kernel.} 
Groups Geom. Dyn. {\bf 17} (2023), 1149--1192.

\bibitem{EH}M. Ershov, S. He. \textsl{On finiteness properties of the Johnson filtrations.} Duke Math. J. {\bf 167}, no. 9 (2018), 1713--1759.



\bibitem{J}D. Johnson. \textsl{Homeomorphisms of a surface which act trivially on homology}. 
Proc. Amer. Math. Soc. {\bf 75} (1979), 119--125.

\bibitem{J1}D. Johnson. \textsl{An abelian quotient of the mapping class group      
${\mathscr I}_g$}. Math. Ann. {\bf 249} (1980), 225--242.


\bibitem{J3}D. Johnson. \textsl{The structure of the Torelli group II: A characterization
of the group generated by twists on bounding curves}. Topology {\bf 24}, n. 2 (1985), 113--126.

\bibitem{J4}D. Johnson. \textsl{The structure of the Torelli group III: The abelianization of ${\mathscr I}$}.
Topology {\bf 24}, n. 2 (1985), 127--144.

\bibitem{M1}S. Morita. \textsl{On the structure of the Torelli group and the Casson invariant.} Topology
{\bf 30}, n. 4 (1991), 603--621.







\end{thebibliography}
\end{document}